\documentclass[10pt]{amsart}
\usepackage[mathcal]{eucal}
\usepackage{amssymb, graphicx, labelfig}

\newtheorem{thm}{Theorem}
\newtheorem{lem}[thm]{Lemma}
\newtheorem{pro}[thm]{Proposition}
\newtheorem{cor}[thm]{Corollary}
\theoremstyle{remark}
\newtheorem{rem}[thm]{Remark}

\newcommand{\db}{/\kern -3pt/}
\newcommand{\St}{\widetilde{S}}

\newcommand{\SL}{\mathrm{SL}}
\newcommand{\PSL}{\mathrm{PSL}}
\newcommand{\Rep}{\mathcal{R}}
\newcommand{\Hom}{\mathrm{Hom}}

\newcommand{\ut}{\widetilde{u}}
\newcommand{\CH}{\mathcal C^{\mathrm{H\ddot ol}}}

\begin{document}

\title[Length functions of Hitchin representations]{Length functions \\ of Hitchin representations}

\author{Guillaume Dreyer}
\address {Department of Mathematics,  University of Notre Dame, 255 Hurley Hall, Notre Dame, IN~46556, U.S.A.}
\email{dreyfactor@gmail.com\\ gdreyer@alumni.usc.edu\\ gdreyer@nd.edu}
\date{\today}
\thanks{This research was partially supported by the grant DMS-0604866 from the National Science Foundation.}

\begin{abstract}
Given a Hitchin representation $\rho \colon \pi_1(S) \to \PSL_n(\mathbb{R})$, we construct $n$ continuous functions $\ell_i^\rho \colon \mathcal \CH(S) \to \mathbb{R}$ defined on the space of H\"older geodesic currents $\CH(S)$ such that, for a closed, oriented curve $\gamma$ in $S$, the $i$--th eigenvalue of the matrix $\rho(\gamma)\in \PSL_n(\mathbb{R})$ is of the form $\pm \mathrm{exp}\, \ell_i^\rho(\gamma)$: such functions generalize to higher rank Thurston's length function of Fuchsian re\-presentations. Identities, diffe\-rentiability properties of these lengths $\ell_i^\rho$, as well as applications to eigenvalue estimates, are also considered.
\end{abstract}

\maketitle

Let $S$ be a closed, connected, oriented surface $S$ of genus $g\geq2$. This article is concerned with homomorphisms $\rho \colon \pi_1(S) \to \PSL_n(\mathbb{R})$ from the fundamental group $\pi_1(S)$ to the Lie group $\PSL_n(\mathbb{R})$ (equal to the special linear group $\SL_n(\mathbb{R})$ if $n$ is odd, and to  $\SL_n(\mathbb{R})/\{\pm \mathrm{Id}\}$ if $n$ is even), and more precisely with elements lying in \emph{Hitchin components} $\mathrm{Hit}_n(S)$ of the $\PSL_{n}(\mathbb{R})$--character variety
$$
\Rep_{\PSL_{n}(\mathbb{R})} (S)=\Hom\bigl(\pi_{1}(S),\PSL_{n}(\mathbb{R})\bigr) \db \PSL_{n}(\mathbb{R})
$$
identified by N.~Hitchin \cite{Hit}. Here, the ``double bar'' sign indicates that the precise definition of the character variety $\Rep_{\PSL_{n}(\mathbb{R})}(S) $ requires that the quotient be taken in the sense of geometric invariant theory \cite{Mum}; however, for the component $\mathrm{Hit}_n(S)$ that we are interested in, this quotient construction coincides with the usual topological quotient. 

A \emph{Hitchin component} $\mathrm{Hit}_n(S)$ is defined as a component of $\Rep_{\PSL_{n}(\mathbb{R})}(S)$ that contains some $n$--\emph{Fuchsian representation}, namely some homomorphism $\rho \colon \pi_1(S) \to \PSL_n(\mathbb{R})$ of the form
$$
\rho=\iota\circ r
$$
where: $r\colon \pi_1(S) \to \PSL_2(\mathbb{R})$ is a discrete, injective homomorphism; and $\iota \colon \PSL_2(\mathbb{R})$ $ \to \PSL_n(\mathbb{R})$ is the preferred homomorphism defined by the $n$--dimensional, irreducible representation of $\SL_2(\mathbb{R})$ into $\SL_n(\mathbb{R})$. These components $\mathrm{Hit}_n(S)$ were singled out by N. Hitchin \cite{Hit} who first suggested the interest in studying their elements. We shall refer to elements of $\mathrm{Hit}_n(S)$ as \emph{Hitchin representations}.

Motivations for studying Hitchin representations find their origin in the case where $n=2$. Hitchin components $\mathrm{Hit}_2(S)$ then coincide with \emph{Teichm\"uller components} $\mathcal{T}(S)$ of $\Rep_{\PSL_{2}(\mathbb{R})}(S)$, whose elements, known as \emph{Fuchsian representations}, are of particular interest as they correspond to conjugacy classes of holonomies of marked hyperbolic structures on the surface $S$. In addition, every  element of $\mathcal{T}(S)$ is a discrete, injective homomorphism, and reversely, any such homomorphism lies in some component $\mathcal{T}(S)$ \cite{We,Mar}. It is a result due to W.~Goldman \cite{Gol} that $\Rep_{\PSL_{2}(\mathbb{R})}(S)$ possesses exactly two Tei\-chm\"uller components $\mathcal{T}(S)$; each of these components $\mathcal{T}(S)$ is known to be homeomorphic to $\mathbb{R}^{6g-6}$ \cite{Th1, FLP}. 

In his foundational paper \cite{Hit}, Hitchin proved that, in the case where $n \geq 3$, there is one or two Hitchin components $\mathrm{Hit}_n(S)$ in $\Rep_{\PSL_{n}(\mathbb{R})}(S)$ according to whether $n$ is odd or even, and a beautiful result of Hitchin is that each of these components $\mathrm{Hit}_n(S)$ is homeomorphic to $\mathbb{R}^{(2g-2)(n^2-1)}$. Hitchin's proof is based the theory of Higgs bundles, and as observed by Hitchin, this geometric analysis framework offers no information about the geometry of the elements of $\mathrm{Hit}_n(S)$. The first geometric result about Hitchin representations is to due to S. Choi and W. Goldman \cite{ChGol} who showed that, for $n=3$, the Hitchin component $\mathrm{Hit}_3(S)$ parametrizes the deformation space of \emph{real convex projective structures} on the surface $S$. As a consequence of their work, they showed the faithfulness and the discreetness for the elements of $\mathrm{Hit}_3(S)$.

About a decade ago, F. Labourie \cite{La1} (see also \cite{Gui, GuiW}) proved the following result.

\begin{thm}[Labourie \cite{La1}]
\label{thm:Labourie}
Let $\rho \colon  \pi_{1}(S) \rightarrow \PSL_{n}(\mathbb{R})$ be a Hitchin representation. Then  $\rho$ is discrete and injective. In addition, the image $\rho(\gamma) \in \PSL_n(\mathbb{R})$ of any nontrivial $\gamma \in \pi_{1}(S)$ is diagonalizable, its eigenvalues are all real with distinct absolute values. 
\end{thm}
 
The above statement comes as a consequence (among others) of a remarkable Anosov property for Hitchin representations discovered by Labourie \cite{La1}. More precisely, let $\rho \colon \pi_1(S)\to \PSL_n(\mathbb{R})$ be a Hitchin representation that lifts to $\rho\colon \pi_1(S)\to \SL_n(\mathbb{R})$; consider the flat, twisted $\mathbb{R}^n$--bundle $T^1S\times_\rho \mathbb{R}^n=T^1S\times \mathbb{R}^n / \pi_1(S) \to T^1S$, where $T^1S$ is the unit tangent bundle of $S$; let $(G_t)_{t\in \mathbb{R}}$ on $T^1S\times_\rho \mathbb{R}^n$ be the flow that lifts the geodesic flow $(g_t)_{t\in\mathbb{R}}$ on $T^1S$ via the flat connection. The total space $T^1S\times_\rho\mathbb{R}^n$ splits as a sum of line subbundles $V_1\oplus\cdots\oplus V_n$ with the property that each line subbundle $V_i\to T^1S$ is invariant under the action of the flow $(G_t)_{t\in \mathbb{R}}$. In addition, the action of the flow $(G_t)_{t\in \mathbb{R}}$ is \emph{Anosov} in the following sense: pick a Riemannian metric $\left \Vert \ \right \Vert$ on $T^1S\times_{\rho} \mathbb{R}^n\to T^1S$; there exist some constants $A\geq 0$ and $a>0$ such that, for every $u\in T^1S$, for every unit vectors $X_i(u)\in V_i(u)$ and $X_j(u)\in V_j(u)$, for every $t>0$,
\begin{align*}
\text{if $i>j$, }&\frac{\left \Vert G_t X_j(u)\right \Vert_{g_t(u)}}{\left \Vert G_t X_i(u)\right \Vert_{g_t(u)}} \leq Ae^{-at};\\
\text{if $i<j$, }&\frac{\left \Vert G_{-t} X_j(u)\right \Vert_{g_{-t}(u)}}{\left \Vert G_{-t} X_i(u)\right \Vert_{g_{-t}(u)}} \leq Ae^{-at}.
\end{align*}

\subsection*{Results} Given a Hitchin representation $\rho\colon \pi_1(S) \to \PSL_n(\mathbb{R})$, our main result uses Labourie's dynamical framework to define a family of $n$ \emph{length functions} $\ell^\rho_i$ associated to $\rho$; these length functions extend to Hitchin representations \emph{Thurston's length function} of Fuchsian representations in the Teichm\"uller space $\mathcal{T}(S)$.    

In \cite{Th1}, Thurston considers the space of \emph{measured laminations} $\mathcal{ML}(S)$ of $S$, that is a certain completion of the set of all isotopy classes of simple, closed, unoriented curves in $S$. He then associates to a Fuchsian representation $r \colon \pi_1(S)\to \PSL_2(\mathbb{R})$ in $\mathcal{T}(S)$ a continuous, homogeneous function
$
\ell^r : \mathcal {ML}(S) \to \mathbb{R}
$
such that, for every simple, closed, unoriented curve $\gamma \subset S$, 
$$
\ell^r(\gamma ) = \frac{1}{2} \log \left |\lambda^r_1(\gamma) \right | 
$$
where $\left|\lambda^r_1(\gamma)\right|$ is the largest absolute value of the eigenvalues of $r(\gamma)\in \PSL_2(\mathbb{R})$. Geometrically,  $r \colon \pi_1(S) \to \mathrm{PSL}_2(\mathbb{R})$ is the holonomy of a marked hyperbolic structure $m$ on the surface $S$; the number $\ell^r(\gamma)$ is then the length of the unique, simple, closed, unoriented $m$--geodesic in $S$ that is freely homotopic to the simple, closed, unoriented curve $\gamma$. This length function $\ell^r \colon \mathcal {ML}(S) \to \mathbb{R}$ has proved to be a fundamental tool in the study of $2$ and $3$--dimensional hyperbolic manifolds.

Thurston's length function $\ell^r$ was later extended by F.~Bonahon  \cite{Bon1, Bon2} to the larger space of \emph{measure geodesic currents} $\mathcal C(S)$ of $S$, which is a certain completion of the set of all isotopy classes of closed, oriented curves in $S$. Later, Bonahon also developed in \cite{Bon3} a differential calculus for measured laminations, that is based on \emph{H\"older geodesic currents}. In particular, he obtains differentiability pro\-perties for Thurston's original function $\ell^r \colon \mathcal {ML}(S) \to \mathbb{R}$ by continuously extending $\ell^{r} \colon\mathcal{ML}(S)\to\mathbb{R}$ to the space of H\"older geodesic currents $\CH(S)$ of $S$. 

We generalize these constructions in the case where $\rho\colon \pi_1(S) \to \PSL_n(\mathbb{R})$ is a Hitchin representation. Let $\gamma\in \pi_1(S)$ be nontrivial element; by Theorem~\ref{thm:Labourie}, the eigenvalues $\lambda^{\rho}_i(\gamma)$ of the matrix $\rho(\gamma)\in \PSL_n(\mathbb{R})$ can be indexed so that 
$$
|\lambda^{\rho}_1(\gamma) |> |\lambda^{\rho}_2(\gamma)| > \dots > |\lambda^{\rho}_n(\gamma)|. 
$$

\begin{thm} \textsc{(Length functions)}
\label{thm:MainThm}
Let $\rho \colon\pi_1(S) \to \PSL_n(\mathbb{R})$ be a Hitchin representation, and let $\CH(S)$ be the vector space of H\"older geodesic currents. For every $i=1$, $2$, $\dots$ , $n$,  there exists a continuous, linear function 
$$
\ell^\rho_{i} \colon \CH(S) \rightarrow \mathbb{R}
$$ 
such that, for every  closed, oriented curve $\gamma\subset S$, $\ell^\rho_i(\gamma) = \log |\lambda^{\rho}_i(\gamma)|$. This continuous extension is unique on the space of measure geodesic currents $\mathcal{C}(S)\subset \CH(S)$.
\end{thm}

In addition, let $\mathfrak{R} \colon T^1S \rightarrow T^1S$ be the \emph{orientation reversing involution}, namely $\mathfrak{R}$ is the map defined by $\mathfrak{R}(u)=-u$, where $u\in T^1_{x}S$. For every H\"older geodesic current $\alpha\in \CH(S)$, $\mathfrak{R}^*\alpha$ is the \emph{pullback current} of $\alpha$ under the involution $\mathfrak{R}$.

\begin{thm} \textsc{(Identities)}
\label{thm:properties}
For every H\"older geodesic current $\alpha\in \CH(S)$, 
\begin{enumerate}
\item $\sum_{i=1}^{n}\ell^{\rho}_{i}(\alpha)=0$;
\item
\label{OrientationReversing}
 $\ell_i^{\rho}(\mathfrak{R}^*\alpha)=-\ell_{n-i+1}^{\rho}(\alpha)$.
\end{enumerate}
\end{thm}

The above two identities are suggested by the case where $\alpha\in \CH(S)$ is a closed, oriented curve $\gamma\in\pi_1(S)$. Indeed, since $\rho(\gamma)\in \PSL_n(\mathbb{R})$, $\sum_{i=1}^{n}\log |\lambda^{\rho}_i(\gamma)|=0$. Moreover, as a consequence of our indexing conventions, $\lambda^{\rho}_i(\gamma^{-1})=1/\lambda^{\rho}_{n-i+1}(\gamma)$, and thus $\log |\lambda^{\rho}_i(\gamma^{-1})|=-\log |\lambda^{\rho}_{n-i+1}(\gamma)|$. 

The continuity property of Theorem~\ref{thm:MainThm} is the fundamental feature of the length functions $\ell_i^{\rho}:\CH(S)\to\mathbb{R}$. As an application of this continuity, we prove the two following results; the proofs use the full force of H\"older geodesic currents.

\begin{thm} \textsc{(Tangentiability)}
\label{Tangentiable}
The functions $\ell^\rho_{i} \colon \CH(S)\to \mathbb{R}$ restrict to functions ${\ell^\rho_{i}}_{| \mathcal{ML}(S)} \colon \mathcal{ML}(S)\to \mathbb{R}$ that are tangentiable, namely, if $(\alpha_{t})_{t\geq0} \subset \mathcal{ML}(S)$ is a smooth $1$--parameter family of measured laminations with tangent vector $\dot{ \alpha_0}=\frac{d}{dt^+}{\alpha_{t}}_{|t=0}$ at $\alpha_0$, then
$$
\frac{d}{d t^+}{\ell_i^{\rho}(\alpha_{t})}_{|t=0}=\ell_i^{\rho}(\dot\alpha_0).
$$
\end{thm} 
Finally, we prove the following asymptotic estimate for the eigenvalues of a Hitchin representation. 
\begin{thm}\textsc{(Eigenvalue estimate)}
\label{cor:application}
Let $\rho\colon \pi_1(S) \to \PSL_n(\mathbb{R})$ be a Hitchin representation, and let $\alpha$, $\beta\in \pi_1(S)$. For every $i=1$, $\ldots$ , $n$,  the ratio
$$
\frac{\lambda^\rho_i(\alpha^m \beta)}{\lambda^\rho_i(\alpha)^m}
$$
has a finite limit as $m$ tends to $\infty$. This limit is equal to $e^{\ell_i^\rho(\dot\alpha)}$ for a certain H\"older geodesic current  $\dot\alpha\in\CH(S)$.  
\end{thm}


\subsection*{Remarks}

In \cite{Dr1}, we extend to Hitchin representations \emph{Thurston's cataclysm deformations}, which themselves generalize (left) earthquake deformations of hyperbolic structures on surfaces \cite{Th1,Th2}. Given a Hitchin representation $\rho\colon \pi_1(S)\to \PSL_n(\mathbb{R})$, we study various geometric aspects of these cataclysms, and prove a variational formula for the associated length functions $\ell^\rho_i$.

Another motivation for introducing length functions associated to a Hitchin representation is part of the development of a new system of coordinates on Hitchin components $\mathrm{Hit}_n(S)$. In \cite{Hit},  Hitchin showed that $\mathrm{Hit}_n(S)$ is diffeomorphic to $\mathbb{R}^{(2g-2)(n^2-1)}$; his parametrization is based on Higgs bundle techniques, and in particular requires the initial choice of a complex structure on $S$. In a joint work with F. Bonahon \cite{BonDr1,BonDr2}, we construct a geometric, real analytic parametrization of Hitchin components $\mathrm{Hit}_n(\mathbb{R}^n)$. One feature of this parametrization is that it is based on topologi\-cal data only. In essence, our coordinates are an extension of Thurston's shearing coordinates \cite{Th2,Bon1} on the Teichm\"uller space $\mathcal{T}(S)$, combined with Fock-Goncharov's coordinates on moduli spaces of positive, framed, local systems on a punctured surface \cite{FoGo}. In particular, the length functions $\ell^\rho_{i}$ play a crucial r\^ole in analyzing the image of this parametrization.

\subsection*{Acknowledgments} 

I would like to thank my advisor, Francis Bonahon, for his invaluable guidance and support. I also thank the referee for his careful reading of the first draft of this article along with numerous remarks and suggestions. This research was partially supported by the grant DMS-0604866 from the National Science Foundation.



\section{The eigenbundles of a Hitchin representation}
\label{sect:LineBundles}

Our construction makes great use of the machinery developed in \cite{La1}, we thus begin with reviewing some of Labourie's framework. It is convenient to endow the surface $S$ with an arbitrary hyperbolic metric $m_0$. It induces a $m_0$--geodesic flow $(g_t)_{t\in \mathbb{R}}$ on the unit tangent bundle $T^1S$ of $S$; we refer to the associated orbit space as the $m_0$--geodesic foliation $\mathcal{F}$ of $T^1S$. 

Let $\rho\colon \pi_1(S) \to \PSL_n(\mathbb{R})$ be a Hitchin representation. Since $\rho$ lies in the same component as some $n$--Fuchsian representation, it lifts to a representation valued in $\SL_n(\mathbb{R})$, that we still denote by $\rho\colon \pi_1(S) \to \SL_n(\mathbb{R})$; see \cite{Gol} for details. Consider the flat twisted $\mathbb{R}^n$--bundle
$$
T^1 S \times_\rho \mathbb{R}^n = T^1\widetilde S \times \mathbb{R}^n / \pi_1(S)
$$
where: $\widetilde S$ is the universal cover of $S$; and where the action of $\pi_1(S)$ is defined by the property that, for every $ \gamma \in \pi_{1}(S)$, for every  $(u,X) \in T^{1}\widetilde{S}\times\mathbb{R}^{n}$, $\gamma(u,X)=(\gamma u, \rho(\gamma)X)$. Let $(G_{t})_{t\in \mathbb{R}}$ be the flow on the total space $T^{1}{S}\times_{\rho}\mathbb{R}^{n}$ that lifts the geodesic flow $(g_t)_{t\in \mathbb{R}}$ on $T^{1}S$ via the flat connection; here, the ``flatness'' condition means that, if one looks at the situation in the universal cover, the lift $(\widetilde{G}_{t})_{t\in \mathbb{R}}$ acts on $T^1\widetilde{S}\times \mathbb{R}^n$ as the geodesic flow $(\widetilde{g})_{t\in \mathbb{R}}$ on the first factor, and trivially on the second factor.  We shall refer to $T^1 S \times_\rho \mathbb{R}^n\to T^1S$ as the \emph{associated, flat $\mathbb{R}^n$--bundle of the Hitchin representation} $\rho\colon \pi_1(S) \to \PSL_n(\mathbb{R})$.

For every nontrivial $\gamma\in \pi_1(S)$, index the eigenvalues $\lambda^{\rho}_i(\gamma)$ of $\rho(\gamma)\in \PSL_{n}(\mathbb{R})$ as in Theorem~\ref{thm:MainThm} so that
$$
|\lambda^{\rho}_1(\gamma) |> |\lambda^{\rho}_2(\gamma)| > \dots > |\lambda^{\rho}_n(\gamma)|. 
$$

The key tool underlying Labourie's analysis is the following decomposition. 

\begin{thm}[Labourie \cite{La1}]\textsc{(Eigenbundle decomposition)}
\label{thm:LineBundles}
The associated, flat $\mathbb{R}^n$--bundle $p \colon T^{1}S\times_{\rho} \mathbb{R}^{n}\rightarrow T^{1}S$ splits as a sum of $n$ line subbundles $V_{1}\oplus \cdots \oplus V_{n}$ that satisfy the following properties:
\begin{enumerate}
\item Each line subbundle $V_i\to T^1S$ is invariant under the flow $(G_{t})_{t\in \mathbb{R}}$;
\item 
\label{Eigenline}
If $u\in T^1S$ is fixed by $g_{t_0} \colon T^1S \to T^1S$ for some $t_0>0$, and if $\gamma\in \pi_1(S)$ represents the corresponding closed orbit of the geodesic flow, then the lift $G_{t_0}$ acts on the fibre $p^{-1}(u)=V_{1}(u)\oplus \cdots \oplus V_{n}(u)$ by multiplication by $1/\lambda^{\rho}_i(\gamma)$ on the line $V_i(u)$;
\item
\label{LineRegularity}
Each line $V_i(u)$ depends smoothly on $u\in T^1S$ along the leaves of the geodesic foliation $\mathcal{F}$, and is transversally H\"older continuous.
\end{enumerate}
\end{thm}

The terminology \emph{eigenbundle decomposition} is motivated by the property (\ref{Eigenline}) that we can make more precise as follows. For every $i=1$, $\ldots$ , $n$, let $\widetilde{V}_i\to T^1\St$ that lifts the line subbundle $V_i\to T^1S$. Let $(\widetilde g_t)_{t\in \mathbb{R}}$ be the lift on the universal cover $T^1\St$ of the geodesic flow $(g_t)_{t\in \mathbb{R}}$ on $T^1S$. Let $u\in T^1S$ that is fixed by $g_{t_0} \colon T^1S\to T^1S$ for some $t_0>0$. Let $\ut \in T^1\St$ that lifts $u\in T^1S$, and let $\gamma\in \pi_1(S)$ be the (unique) nontrivial element such that $\widetilde{g}_{t_0}(\ut)=\gamma \ut$. Because of the flat connection, and the invariance of the line subbundle $V_i\to T^1S$ under the flow $(G_t)_{t\in \mathbb{R}}$, we have that $\widetilde{V}_i(\widetilde{g}_{t_0}(\ut))=\widetilde{V}_i(\ut)$ as lines of $\mathbb{R}^n$. In addition, $\widetilde{V}_i(\widetilde{g}_{t_0}(\ut))=\widetilde{V}_i(\gamma\ut)=\rho(\gamma)\widetilde{V}_i(\ut)$ (it is the equivariance property for the lift $\widetilde{V}_i\to T^1\St$). Hence $\widetilde{V}_i(\ut)\subset \mathbb{R}^n$ is an eigenspace for $\rho(\gamma)\in \PSL_n(\mathbb{R})$, and $\rho(\gamma)$ is diagonalizable. Finally, note that, for every $X\in \mathbb{R}^n$, $(\widetilde{g}_{t_0}(\ut), \widetilde{G}_{t_0}X)=(\gamma \ut, X)$ identifies in the quotient with $(\ut, \rho(\gamma)^{-1}X)$. Therefore,  the lift $G_{t_0}$ acts on the line $V_i(u)$ by multiplication by $1/\lambda^{\rho}_i(\gamma)$.

As a consequence of the above discussion, we make the following observation, that we state as a lemma for future reference.
\begin{lem}
\label{lem:LineReversing}
Let  $\widetilde{V}_i\to T^1\St$ and $\widetilde{V}_{n-i+1}\to T^1\St$ that lift the line subbundles $V_i\to T^1S$ and $V_{n-i+1}\to T^1S$, respectively. For every $u\in T^1S$ that lifts to $\ut\in T^1\widetilde S$, the fibres $\widetilde{V}_i(\ut)$ and $\widetilde{V}_{n-i+1}(-\ut)$ coincide as lines of $\mathbb{R}^n$.
\end{lem}

\begin{proof}
When $u$ lies in a closed leaf  of the geodesic foliation $\mathcal{F}$, the assertion immediately comes as a consequence of the property (\ref{Eigenline}) of Theorem~\ref{thm:LineBundles}, and from our indexing conventions for the eigenvalues of $\rho(\gamma)\in \PSL_n(\mathbb{R})$ when $\gamma\in \pi_1(S)$. The general case then follows from the latter by density of closed leaves in $T^1S$.
\end{proof}

The existence of an eigenbundle decomposition for the associated, flat $\mathbb{R}$--bundle $T^{1}S\times_{\rho}\mathbb{R}^{n}\to T^1S$ of a Hitchin representation $\rho\colon \pi_1(S) \to \PSL_n(\mathbb{R})$ as in Theorem~\ref{thm:LineBundles} is a consequence of  \emph{Labourie's Anosov property} for Hitchin representations; \cite{La1, Gui, GuiW, Dr1} for additional details.

\section{The length functions of a Hitchin representation}

\subsection{H\"older geodesic currents}
\label{sect:Currents} 
Before tackling the construction of the length functions $\ell^\rho_i  \colon \CH(S) \to \mathbb{R}$, we need to remind the reader of the definition of H\"older geodesic currents; see \cite{Bon2, Bon3, Bon4} for details.

Let  $(X,d)$ be metric space.  A \emph{H\"older distribution} $\alpha$ is a continuous, linear functional on the space of compactly supported, H\"older continuous functions $\phi \colon X\to \mathbb{R}$. A special case of H\"older distributions are \emph{positive Radon measures}, which are linear functionals on the space of compactly supported, continuous functions, and associate to a nonnegative function a nonnegative number. 

The unit tangent bundle $T^1S$ is a $3$--dimensional manifold, and the orbits of the $m_0$--geodesic flow $(g_t)_{t\in \mathbb{R}}$ define a $1$--dimensional foliation $\mathcal F$ of $T^1S$ called its $m_0$--\emph{geodesic foliation}. It turns out that, whereas the geodesic flow depends of the auxiliary me\-tric $m_0$ that we have chosen on $S$, the geodesic foliation does not. Indeed, if another negatively curved metric $m'$ defines a geodesic foliation $\mathcal F'$, there is a homeomorphism of $T^1S$ that sends $\mathcal F$ to $\mathcal F'$. In addition, this homeomorphism can be chosen to be isotopic to the identity, and H\"older bi-continuous; see \cite{Br, Ghys, Gro1, Gro2} for details.

A \emph{H\"older geodesic current} $\alpha$ on $S$ is a transverse H\"older distribution for the geodesic foliation $\mathcal F$, namely $\alpha$ assigns a H\"older distribution $\alpha_D$ on every surface $D\subset T^1S$ transverse to $\mathcal F$. This assignment is invariant under restriction: for any subsurface $D'\subset D$, ${\alpha_D}_{|D'}=\alpha_{D'}$; and is homotopy invariant: for any (H\"older) homotopy $h \colon D\to D''$ from $D$ to another transverse surface $D''$ that preserves $\mathcal F$, $\alpha_D=h^*\alpha_{D''}$ ($h^*\alpha_{D''}$ is the pullback of $\alpha_{D''}$ by $h$).

When the transverse H\"older distribution $\alpha$ is actually a measure $\alpha_D$ for every surface $D\subset T^1S$ transverse to $\mathcal{F}$, the corresponding H\"older geodesic current is a \emph{measure geodesic current} of $S$. Let $\CH(S)$ and $\mathcal C(S)$ be respectively the space of H\"older geodesic currents, and the space of measure geodesic currents. Note that $\CH(S)$ is a (real) vector space, and $\mathcal{C}(S)$ is stable under positive scalar multiplication.

The space of H\"older geodesic currents $\CH(S)$ is endowed with the \emph{weak-* topolo\-gy}, namely the weakest topology for which, for every surface $D\subset T^1S$ transverse to  $\mathcal{F}$, the linear function $\varphi_{D} \mapsto \alpha_D(\varphi_D)$ is continuous, where $\varphi_D$ ranges over all compactly supported, H\"older continuous functions on the surface $D$. 

A typical example of measure geodesic current is provided by the free homotopy class of a closed, oriented curve $\gamma \subset S$. Let $k\geq 0$ be the largest integer such that $\gamma$ is homotopic to a $k$--multiple $\gamma_1^k$ of a closed curve $\gamma_1$.  The homotopically primitive curve $\gamma_1$ is freely homotopic to a unique closed, oriented geodesic, which itself corresponds to a closed leaf $\gamma^*_1$ of the geodesic foliation $\mathcal F$. In particular, we associate to $\gamma_1$ the transverse $1$--weighted Dirac measure for $\mathcal F$ defined by the closed orbit $\gamma^*_1$: for every surface $D$ transverse to $\mathcal F$, the measure ${\gamma_1}_D$ is the counting measure at the intersection points $D\cap \gamma_1^*$. As a result, we associate to $\gamma=\gamma_1^k$  $k$--times  the transverse $1$--weighted Dirac measure associated to $\gamma_1$. Hence the following embedding
$$
\{\text{closed, oriented curves in } S\}/\text{homotopy} \,\subset \mathcal C(S) \subset \CH(S). 
$$
In addition, the set of positive real, linear combinations of multiples of homotopy classes of closed, oriented curves are dense in $\mathcal C(S)$.

Finally, note that to a closed, \emph{unoriented} curve $\bar{\gamma}$ in $S$ corresponds two closed leaves $\gamma^*$ and $\big (\mathfrak{R}(\gamma)\big )^*$ of the geodesic foliation $\mathcal F$, and thus two measure geodesic currents $\gamma$ and $\mathfrak{R}^*\gamma\in \mathcal{C}(S)$ (as in Theorem~\ref{thm:properties}, $\mathfrak{R} \colon T^1S\to T^1S$ denotes the orientation reversing involution, and $\mathfrak{R}^*\colon \CH(S)\to \CH(S)$ is the pullback involution induced by $\mathfrak{R}$). Therefore, the set of  closed, \emph{unoriented} curves in $S$ can be formally embedded in $\CH(S)$ as follows: for every  closed, unoriented curve $\bar{\gamma}\subset S$, 
\begin{eqnarray}
\label{unoriented}
\bar{\gamma}=\frac{1}{2}\gamma+\frac{1}{2}\mathfrak{R}^*\gamma \in \CH(S).
\end{eqnarray}
In particular, the space of measured laminations $\mathcal{ML}(S)$, that is defined as  the closure in $\mathcal{C}(S)$  of the set of  positive real multiples of homotopy classes of  \emph{simple, closed, unoriented} curves in $S$ \cite{Bon1, Bon2}, corresponds to the closure in $\mathcal{C}(S)$ of the set of positive real, linear combinations of elements of the above form (\ref{unoriented}). 

\subsection{Lengths of closed, oriented curves}
\label{sect:LengthsCurves} 
Let $\rho\colon \pi_1(S) \to \PSL_n(\mathbb{R})$ be a Hitchin representation. For a nontrivial $\gamma \in \pi_1(S)$,  Theorem~\ref{thm:Labourie} shows that the eigenvalues $\lambda^{\rho}_i(\gamma)$ of the matrix $\rho(\gamma)\in \PSL_n(\mathbb{R})$ are all real, and can be indexed so that
$$\left | \lambda^{\rho}_{1}(\gamma) \right | > \left | \lambda^{\rho}_{2}(\gamma) \right | > \cdots>\left |\lambda^{\rho}_{n}(\gamma)\right |.$$
Set $\ell_i^\rho(\gamma) = \log |\lambda^{\rho}_i(\gamma)|$. Note that $\ell_i^\rho(\gamma)$ depends only on the conjugacy class of $\gamma \in \pi_1(S)$, and thus depends only on the free homotopy class of the closed, oriented  curve $\gamma \subset S$. We have $n$ maps 
$$
\ell_i^\rho : \{\text{closed, oriented curves in } S\}/\text{homotopy} \to \mathbb{R}.
$$

\subsection{$1$--forms along the geodesic foliation}
\label{sect:Forms} 

We now construct, for every $i=1$, $\ldots$ , $n$, a $1$--form $\omega_i$ along the leaves of the geodesic foliation $\mathcal F$ of $T^1S$.

Given a Hitchin representation $\rho\colon \pi_1(S) \to \PSL_n(\mathbb{R})$, consider its associated flat $\mathbb{R}^n$--bundle $T^1 S \times_\rho \mathbb{R}^n \to T^1S$ as in \S\ref{sect:LineBundles}. Let $(G_t)_{t\in \mathbb{R}}$ the flow on $T^1 S \times_\rho \mathbb{R}^n$ that lifts the geodesic flow on $T^1S$ via the flat connection. Let $V_1$, $V_2$, $\ldots$ , $V_n\to T^1S$ be the line subbundles of the eigenbundle decomposition of Theorem~\ref{thm:LineBundles}; a fundamental property for each of the line subbundles $V_i\to T^1S$ is to be invariant under the action of the flow $(G_t)_{t\in \mathbb{R}}$. Finally, pick a Riemannian metric $\left\Vert \ \right\Vert$ on $T^1S \times_\rho \mathbb{R}^n \to T^1S$.

Let $\mathcal L$ be a leaf of the geodesic foliation $\mathcal F$. Pick a point $u_0\in \mathcal L\subset T^1S$ and a vector $X_i(u_0)$ in the fibre $V_i(u_0)$ of the line subbundle $V_i\to T^1S$. For every $t$  on a neighborhood of $0$ in $\mathbb{R}$, set 
$$
f_{X_i(u_0)} \big ( g_t(u_0) \big ) = \log \left\Vert G_t X_i(u_0) \right\Vert_{g_{t}(u_{0})}.
$$
The above expression defines a function $f_{X_i(u_0)}\colon I\to \mathbb{R}$, where  $I$ is a neighborhood  of $u_0$ in the leaf $\mathcal{L}$. In addition, the fibre $V_i(u)$ depending smoothly on the point $u\in T^1S$ along the leaves of the geodesic foliation $\mathcal F$, the function $f_{X_i(u_0)}$ is smooth along the leaf $\mathcal  L$. For every $u$ on the same neighborhood $I$ of $u_0$ in the leaf $\mathcal  L$, set 
$$
{\omega_i}(u) = -d_uf_{X_i(u_0)}
$$
where the differential $d_uf_{X_i(u_0)}$ is taken along the leaf $\mathcal  L$. 

\begin{lem}
\label{lem:DiffFormDefined}
Defined as above, $\omega_i$ is a well-defined $1$--form along the leaf $\mathcal{L}\subset \mathcal{F}$.
\end{lem}

\begin{proof}
We must verify that $\omega_{i}$ does not depend on the choices of the point $u_{0}\in I$, and of the vector $X_{i}(u_0) \in V_i(u_0)$.

Since the fibre $V_i(u_0)$ is a line, any other choice $X_i'(u_0)$ for $X_i(u_0)$ is of the form $X_i'(u_0)=cX_i(u_0)$ for some $c\in \mathbb{R}$. Then
$$
f_{X_i'(u_0)} = f_{X_i(u_0)} + \log \left \vert c \right \vert
$$
and $df_{X_i'(u_0)}=df_{X_i(u_0)}$. Hence $\omega_i$ is independent of the choice of $X_i(u_0)\in V_i(u_0)$. 

Let $u_0' = g_{t_0}(u_0) \in I$ be another point, and let $X_i'(u'_0)=G_{t_0}X_i(u_0)\in V_i(u_0')$. Then the functions
$$
f_{u_0', X_i'(u'_0)} = f_{u_0, X_i(u_0)}  
$$
coincide on $I$ since $G_t X_i'(u'_0) = G_{t+t_{0}}X_i(u_0)$, and thus have the same differential along the leaf $\mathcal  L$.
\end{proof}

It follows from Lemma~\ref{lem:DiffFormDefined} that $\omega_i$ is a well-defined $1$--form along the leaves of $\mathcal F$. Moreover, let us also make the following observation regarding the global regularity of the $1$--form $\omega_i$.

\begin{lem}
\label{lem:DiffFormHolder}
The $1$--form $\omega_{i}$ is smooth along the leaves of the geodesic foliation $\mathcal F$, and is transversally H\"older continuous. 
\end{lem}

\begin{proof}
This is an immediate consequence of the regularity property (\ref{LineRegularity}) of Theorem~\ref{thm:LineBundles}. 
\end{proof}

\subsection{Lengths of H\"older geodesic currents}
\label{sect:LengthsCurrents} 

We now make use of the $1$--forms $\omega_i$ to define the lengths $\ell_i^\rho(\alpha)$ of a H\"older geodesic current $\alpha\in \CH(S)$.

Let $\{\mathcal{U}_{j}\}_{j=1, \dots ,m}$ be a finite family of \emph{flow boxes} $\{\mathcal{U}_{j}\}_{j=1,\dots,m}$ that covers the compact, foliated $3$--manifold $T^1S$. By flow box, we mean an open subset $\mathcal U_j\subset T^1S$ such that there exists a diffeomorphism $\mathcal U_j \cong D_j \times (0,1)$, where: $D_j$ is an open subset of $\mathbb{R}^2$; and where, for every $x\in D_j$, the interval $\{x\} \times(0,1)$ corresponds to an arc in a leaf of $\mathcal F$. Let $\{\xi_j\}_{j=1, \dots, m}$ be a partition of unity subordinate to the open covering $\{\mathcal{U}_{j}\}_{j=1,\dots,m}$. By integrating the $1$--form $\xi_j\omega_i$ along the arcs of leaves in $\mathcal U_j$, we define a function $\phi_j \colon D_j \to \mathbb{R}$ by
$$
\phi_j(x) = \int_{\{x\}\times(0,1)} \xi_j \omega_i.
$$
Note that, by Lemma~\ref{lem:DiffFormHolder}, the above function $\phi_j \colon D_j \to \mathbb{R}$ is H\"older continuous with compact support. The H\"older geodesic current $\alpha$ induces a H\"older distribution $\alpha_{D_j}$ on $D_j$, that we shall still denote by $\alpha$ to alleviate notations. Let us denote the evaluation of $\alpha$ at the function $\phi_j$ by
$$
\alpha(\phi_j)=\int_{\mathcal U_j} \xi_j\omega_i \,d\alpha
$$
where the integral notation is suggested by the case where $\alpha$ is a transverse measure for $\mathcal F$. Finally, set
$$
\ell^{\rho}_{i}(\alpha)=\int_{T^{1}S}\omega_{i}\,d\alpha=\sum_{j=1}^m \int_{\mathcal{U}_{j}}\xi_{j}\omega_{i}\,d\alpha.
$$
By the usual linearity arguments, $\ell^{\rho}_{i}(\alpha)$ is independent of the choice of the open covering $\{\mathcal{U}_{j}\}_{j=1,\dots,m}$ and of the partition of unity $\{\xi_{j}\}_{j=1,\dots,m}$.

\begin{thm}
\label{thm:LengthFunctions}
Defined as above, for every $i=1$, $\ldots$ , $n$,
$$
\ell^{\rho}_{i} \colon \CH(S) \to \mathbb{R}
$$
 is a continuous, linear function on the vector space of H\"older geodesic currents $\CH(S)$ that extends the length $\ell^{\rho}_{i}$ of closed, oriented curves of \S\ref{sect:LengthsCurves}. This continuous extension is unique on the space of measure geodesic currents $\mathcal{C}(S)\subset \CH(S)$. In addition, the length $\ell^{\rho}_{i}$ does not depend on the choice of the Riemannian metric $\left\Vert \ \right\Vert$ on the associated, flat $\mathbb{R}^n$--bundle $T^{1}S\times_{\rho}\mathbb{R}^n\to T^1S$ that defines the $1$--form $\omega_i$ of \S\ref{sect:Forms}.
\end{thm}

\begin{proof}[Proof of Theorem~\ref{thm:LengthFunctions}]

We organize the proof in several steps.

\begin{lem}
For every closed, oriented curve $\gamma\subset S$, 
$$
\ell^{\rho}_{i}({\gamma})=\log \left | \lambda^{\rho}_{i}(\gamma) \right |
$$
where: $\ell^{\rho}_{i}({\gamma})$ is the image of the H\"older geodesic current $\gamma\in \CH(S)$ under the function $\ell_i^\rho:\CH(S) \to \mathbb{R}$; and where $\lambda^{\rho}_{i}(\gamma)$ is the $i$--th eigenvalue of $\rho(\gamma)$.
\end{lem}

\begin{proof} We need to return to the definition of the H\"older geodesic current $\gamma \in \CH(S)$.

By homogeneity of the function $\ell_i^\rho$, we can focus attention to the case where the closed, oriented curve $\gamma\subset S$ is homotopically primitive, namely $\gamma$ is not homotopic to a multiple $\gamma_1^k$ of a closed, oriented curve $\gamma_1$ with $k\geq 2$. Thus $\gamma$ determines a closed,  oriented $m_0$--geodesic of $S$, and a closed leaf $\gamma^*$ of the geodesic foliation $\mathcal F$. Identify the closed, oriented curve $\gamma\subset S$ with the H\"older geodesic current $\gamma\in \CH(S)$, which is the transverse $1$--weighted Dirac measure defined by the associated closed leaf $\gamma^*$ (see \S\ref{sect:Currents}).  

By definition of the function $\ell_i^\rho\colon\CH(S) \to \mathbb{R}$,
$$
\ell_i^\rho(\gamma) = \int_{T^1S} \omega_i \,d\gamma = \int_{\gamma^*} \omega_i. 
$$
To compute this integral, pick a point $u_0 \in \gamma^*\subset T^1S$, and a nonzero vector $X_i(u_0) \in V_i(u_0)$ in the fibre of the line subbundle $V_i\to T^1S$. By definition of the $1$--form $\omega_i$ (see \S\ref{sect:Forms}), 
\begin{eqnarray*}
\int_{\gamma^*} \omega_i &=&\int_{0}^{t_\gamma} \frac{d}{ds}\left ( -\log \left \Vert G_s X_i(u_0) \right \Vert_{g_{s}(u_0)}\right )ds\\
&=&\log \left\Vert G_0X_i(u_0)  \right\Vert_{u_0} -  \log \left\Vert G_{t_\gamma}X_i(u_0)  \right\Vert_{g_{t_\gamma}(u_0)}
\end{eqnarray*}
where  $t_{\gamma}$ is the necessary time to go around the closed leaf $\gamma^*\subset \mathcal{F}$ by the geodesic flow $(g_t)_{t\in\mathbb{R}}$, namely $t_\gamma$ is the smallest $t>0$ such that $g_t(u_0) = u_0$. By the property (\ref{Eigenline}) of Theorem~\ref{thm:LineBundles}, $G_{t_\gamma}X_i(u_0) = 1/\lambda^{\rho}_i(\gamma) X_i(u_0)$, and $G_0X_i(u_0)=X_i(u_0)$ since $(G_t)_{t\in \mathbb{R}}$ is a flow, which proves the assertion.
\end{proof}

\begin{lem}
\label{pro:ContHomo}
The function $\ell^{\rho}_{i} \colon \CH(S) \rightarrow \mathbb{R}$ is linear and continuous. Its restriction  ${\ell^{\rho}_{i}}_{|\mathcal{C}(S)} \colon \mathcal{C}(S) \rightarrow \mathbb{R}$ is positively homogeneous, and is the unique, continuous extension to the space of measure geodesic currents $\mathcal{C}(S)\subset \CH(S)$ for the length $\ell^{\rho}_{i}$ of closed, oriented curves of \S\ref{sect:LengthsCurves}. 
\end{lem}

\begin{proof} By construction, for every H\"older geodesic current $\alpha\in \CH(S)$,
$
\ell^{\rho}_{i}(\alpha)=\sum_{j=1}^m \alpha(\phi_{j})
$.
The linearity and homogeneity are immediate. The continuity follows from the definition of the weak-* topology of $\CH(S)$. Finally, since the set of  positive real, linear combinations of multiples of closed, oriented curves are dense in the space of measure geodesic currents $\mathcal C(S)$, the restriction of this continuous extension to $\mathcal C(S)$ is unique. 
\end{proof}

\begin{lem}
\label{lem:Metric}The length $\ell^{\rho}_{i}(\alpha)$ is independent of the choice of the Riemannian metric $\left\Vert \ \right\Vert$ on the associated, flat $\mathbb{R}^n$--bundle $T^{1}S\times_{\rho}\mathbb{R}^n\to T^1S$.
\end{lem}

\begin{proof}
Let  $\left\Vert \ \right\Vert'$ be another Riemannian metric on the associated, flat $\mathbb{R}^n$--bundle $T^{1}S\times_{\rho}\mathbb{R}^n$; it induces another $1$--form $\omega_i'$ along the leaves of the geodesic foliation $\mathcal F$. Since the line $V_i(u)$ depends smoothly on $u\in T^1S$ along the leaves of $\mathcal F$, there exists a positive function $f \colon T^1S \to \mathbb{R}$, smooth along of the leaves of $\mathcal F$, such that, for every $X_{i}(u)\in V_i(u)$, $\left\Vert X_{i}(u) \right\Vert_u = f(u) \left\Vert X_{i}(u) \right\Vert'_u$. As a result, 
$$
\omega_i' = \omega_i - d\log f
$$
which implies that 
$$
\int_{T^{1}S}\omega_{i}^{'}\,d\alpha = \int_{T^{1}S}\omega_{i}\,d\alpha - \int_{T^{1}S} d\log f\,d\alpha.
$$
Since $\sum_{j=1}^m\xi_{j}=1$,
$$
\int_{T^{1}S}{d}\log f\, d\alpha=\sum_{j=1}^m\int_{\mathcal U_j}{d}(\xi_{j}\log f)d\alpha.
$$
For our notation conventions, for every $j=1$, $\ldots$ , $m$,
$$
\int_{\mathcal U_j}{d}(\xi_{j}\log f)d\alpha=\alpha_{D_j} \left ( \int_{ \{ x \} \times (0,1)}{d}(\xi_{j}\log f) \right )
$$
where $\mathcal{U}_j \cong D_j \times (0,1)$. By Stokes's, $\int_{ \{ x \} \times (0,1)}{d}(\xi_{j}\log f)=0$, which proves the assertion.
\end{proof}
This achieves the proof of Theorem~\ref{thm:LengthFunctions}.
\end{proof}

\section{Properties of the length functions}

\subsection{Symetries of the lengths}
We now prove the two properties of Theorem~\ref{thm:properties}.

\begin{pro}
\label{pro:Sum}
For every H\"older geodesic current $\alpha\in \CH(S)$, 
$$
\sum_{i=1}^{n}\ell^{\rho}_{i}(\alpha)=0.
$$
\end{pro}

\begin{proof}
Endow each fibre $\{ \widetilde{u}\}\times\mathbb{R}^n$ of the trivial bundle $T^1\widetilde{S}\times\mathbb{R}^n\to T^1\St$ with the canonical volume form $\sigma=dx_1\wedge dx_2\wedge\ldots\wedge dx_n$ of $\mathbb{R}^n$. Recall that $\pi_1(S)$ acts on $T^1\widetilde{S}\times\mathbb{R}^n$ via the diagonal action. Since each $\rho(\gamma)$ is in $ \SL_n(\mathbb{R})$, the form $\sigma$ is invariant under the action of $\pi_1(S)$. In addition, because of the flat connection, the lift $(\widetilde{G}_t)_{t\in \mathbb{R}}$ on $T^1\widetilde{S}\times\mathbb{R}^n$ of the geodesic flow $(\widetilde{g}_t)_{t\in \mathbb{R}}$ on $T^1\St$ acts trivially on the factor $\mathbb{R}^n$ of $T^1\widetilde{S}\times\mathbb{R}^n$, and consequently preserves $\sigma$. As a result, $\sigma$ descends to a well-defined $G_t$--invariant volume form on the fibres of the bundle $T^1S\times_{\rho}\mathbb{R}^n$.

Recall that the length functions $\ell^{\rho}_{i} \colon \CH(S)\to \mathbb{R}$ are independent of the choice of the Riemannian metric $\left\Vert \ \right \Vert$ on the bundle $T^1S\times_{\rho}\mathbb{R}^n$. Without loss of generality, we can arrange that the line subbundles $V_{i}$ are orthogonal for $\left\Vert \ \right \Vert$, and that the volume form defined by $\left\Vert \ \right \Vert$ coincides with the volume form $\sigma$.

By definition of the $1$--form $\omega_i$, for every  $u\in T^1S$, for every vector  $X_i(u)\in V_i(u)$, 
\begin{align*}
\sum_{i=1}^{n}\omega_{i}(u)&=\sum_{i=1}^{n} {\frac{d}{dt} \left ( - \log  \left \Vert G_t X_i(u) \right \Vert_{g_{t}(u)}\right )dt}_{|t=0}\\
&=-{\frac{d}{dt}\log \left ( \prod_{i=1}^{n}\left \Vert G_t X_i(u) \right \Vert_{g_{t}(u)}\right )dt}_{|t=0}\\
&=-{\frac{d}{dt}\log \left ( \sigma_{g_t(u)} \big (G_t X_1(u), G_t X_2(u), \ldots , G_t X_n(u)\big ) \right )dt}_{|t=0}\\ 
&=-{\frac{d}{dt}\log \left ( \sigma_{u} \big  ( X_1(u),X_2(u), \ldots ,X_n(u) \big ) \right )dt}_{|t=0}\\ 
&=0
\end{align*}
By integrating, it follows that,  for every H\"older geodesic current $\alpha \in \CH(S)$, $\sum_{i=1}^{n}\ell^{\rho}_{i}(\alpha)=0
$.
\end{proof}

The unit tangent space $T^1S$ comes endowed with a natural, fiberwise involution $\mathfrak{R}\colon T^1S \rightarrow T^1S$, which to $u\in T^1_xS$ associates $\mathfrak{R}(u)=-u$: it is the \emph{orientation reversing involution}. In particular, $\mathfrak{R}$ respects the geodesic foliation $\mathcal F$, and thus induces an involution $\mathfrak{R}^* \colon \CH(S)\to \CH(S)$ defined as follows: for every H\"older geodesic current $\alpha \in \CH(S)$, $\mathfrak{R}^*\alpha$ is the \emph{pullback current} of $\alpha$ under the involution $\mathfrak{R}$, namely, if $\varphi \colon D\to\mathbb{R}$ is H\"older continuous with compact support defined on a transverse surface $D\subset T^1S$, then  $\mathfrak{R}^*\alpha(\varphi)=\alpha(\varphi\circ \mathfrak{R})$. Note that  the restriction to each oriented leaf of the geodesic foliation $\mathcal{F}$ of the involution $\mathfrak{R}$ is orientation reversing.

\begin{pro}
\label{Prop15}
For every H\"older geodesic current  $\alpha\in\CH(S)$,
$$
\ell_i^{\rho}(\mathfrak{R}^*\alpha)=-\ell_{n-i+1}^{\rho}(\alpha). 
$$
\end{pro}

\begin{proof}
The involution $\mathfrak{R}$ acts freely on the unit tangent bundle $T^1S$ with quotient $\widehat{T}^1S=T^1S/\mathfrak{R}$. Thus, it also acts freely on the total space $T^1S\times_{\rho}\mathbb{R}^n$ with quotient a bundle  $\widehat T^1S\times_{\rho}\mathbb{R}^n$. Consider a Riemannian metric $\left\Vert \ \right \Vert$ on $T^1S\times_{\rho}\mathbb{R}^n$ obtained by lifting a Riemannian metric on $\widehat T^1S\times_{\rho}\mathbb{R}^n$. By construction, the Riemannian metric $\left\Vert \ \right \Vert$ is invariant under the involution $\mathfrak{R}$.

On the other hand, let $\widetilde{V}_i\to T^1\St$ and $\widetilde{V}_{n-i+1}\to T^1\St$ that lift the subbundles $V_i\to T^1S$ and $V_{n-i+1}\to T^1S$, respectively. By Lemma~\ref{lem:LineReversing}, for every $\ut\in T^1S$, for every $t\in \mathbb{R}$, the fibres $\widetilde{V}_i(\widetilde{g}_t(\mathfrak{R}(\ut)))$ and $\widetilde{V}_{n-i+1}(\widetilde{g}_{-t}(\ut))$ coincide as lines of $\mathbb{R}^n$. Therefore, $\mathfrak{R}^*\omega_i=\omega_{n-i+1}$. By integrating, it follows that, for every H\"older geodesic current  $\alpha\in\CH(S)$, $\ell_i^{\rho}(\mathfrak{R}^*\alpha)=-\ell_{n-i+1}^{\rho}(\alpha)$; note that a minus sign pops up due to the orientation reversing property of the involution $\mathfrak{R}$.
\end{proof}

\subsection{Differentiability of the lengths}
\label{DiffProperties}

We discuss some regularity properties for the length functions $\ell^\rho_i \colon \CH(S)\to \mathbb{R}$.

Thurston considers in \cite{Th1} the space of measured laminations $\mathcal{ML}(S)$ of $S$, that is a certain completion of the set of all isotopy classes of simple, closed, unoriented  curves in $S$. A fundamental feature of the space $\mathcal{ML}(S)$ is that it is a \emph{piecewise linear} manifold, which is homeomorphic to $\mathbb{R}^{6g-6}$. Therefore, every measured lamination admits tangent vectors, which allows to tackle \emph{tangentiality} properties for functions that are defined on $\mathcal{ML}(S)$. We refer the reader to \cite{PeH,Bon3} for additional details about the $\text{PL}$--structure of $\mathcal{ML}(S)$.

The definition of tangent vectors to $\mathcal{ML}(S)$ is rather abstract and not very convenient in practice. In \cite{Bon3}, Bonahon gives an analytical interpretation for the tangent vectors to $\mathcal{ML}(S)$ as certain H\"older geodesic currents. In particular, a consequence of this work is the following simple criterion.

Recall that the space of measured laminations $\mathcal{ML}(S)$ can be viewed as the closure in $\CH(S)$ of the set of positive real, linear combinations elements of the form  
$$
\bar{\gamma}=\frac{1}{2}\gamma+\frac{1}{2}\mathfrak{R}^*\gamma
$$ 
where: $\gamma \in \CH(S)$ is a closed, oriented curve; and $\mathfrak{R} \colon T^1S\to T^1S$ is the orientation reversing involution; see \S\ref{sect:Currents}. 

\begin{thm}[Bonahon \cite{Bon3}]
\label{criterion}
Let $f \colon \mathcal{ML}(S)\to \mathbb{R}$ be a homogeneous function defined on the space of measured lamination $\mathcal{ML}(S)$. If $f$ admits a continuous, linear extension $f \colon \CH(S)\to \mathbb{R}$ to the vector space of H\"older geodesic currents $\CH(S)$, then $f \colon \mathcal{ML}(S)\to \mathbb{R}$ is \emph{tangentiable}, namely it is differentiable with respect to directions of tangent vectors to $\mathcal{ML}(S)$. 
\end{thm}

For every $i=1$, $\ldots$ , $n$, let ${\ell^\rho_{i}}_{| \mathcal{ML}(S)} \colon \mathcal{ML}(S)\to \mathbb{R}$ be the functions obtained by restricting the length functions $\ell^\rho_i \colon \CH(S)\to \mathbb{R}$ to the space of measured laminations $\mathcal{ML}(S)$. As a straitghforward consequence of the criterion of Theorem~\ref{criterion}, we have:

\begin{cor}
\label{Differentiability}
The functions ${\ell^\rho_{i}}_{| \mathcal{ML}(S)} \colon \mathcal{ML}(S)\to \mathbb{R}$ are tangentiable, namely, if $(\alpha_{t})_{t\geq0} \subset \mathcal{ML}(S)$ is a smooth $1$--parameter family of measured laminations with tangent vector $\dot{ \alpha_0}=\frac{d}{dt^+}{\alpha_{t}}_{|t=0}$ at $\alpha_0$, then
$$
\frac{d}{d t^+}{\ell_i^{\rho}(\alpha_{t})}_{|t=0}=\ell_i^{\rho}(\dot\alpha_0).
$$
\end{cor}

\section{An asymptotic estimate for the eigenvalues}
\label{application}

As another application of the continuity property of the lengths $\ell^\rho_i$, we now prove the following estimate. 

\begin{thm}
\label{pro:application}
 Let $\rho\colon \pi_1(S) \to \PSL_n(\mathbb{R})$ be a Hitchin representation, and let $\alpha$, $\beta\in \pi_1(S)$. For every $i=1$, $\ldots$ , $n$, the ratio
$$
\frac{\lambda^\rho_i(\alpha^m \beta)}{\lambda^\rho_i(\alpha)^m}
$$
has a finite limit as $m$ tends to $\infty$. This limit is equal to $e^{\ell_i^\rho(\dot\alpha)}$, where $\dot\alpha$ is the H\"older geodesic current $\dot\alpha= \lim_{m\to\infty} \alpha^m\beta-m\alpha \in \CH(S)$.
\end{thm}

\begin{proof}
Without loss of generality, we can assume that $\alpha$ is primitive in $\pi_1(S)$. As in \S\ref{sect:Currents}, identify the closed, oriented curves $\alpha$ and $\alpha^m\beta$ with the corresponding closed leaves $\alpha^*$ and ${\alpha^m\beta}^*$ of the geodesic foliation $\mathcal{F}$ of $T^1S$. Endowing these closed leaves with the transverse Dirac measures that they define, we can regard $\alpha$ and $\alpha^m\beta$ as H\"older geodesic currents.

For $m$ large enough, the closed leaf $\alpha^m\beta$ is made up of one piece of uniformly bounded length, and of another piece that wraps $m$ times around $\alpha$. As $m$ tends to $\infty$, this closed leaf converges to the union of the closed orbit $\alpha$ and of an infinite leaf $\alpha^\infty\beta$ of the geodesic foliation whose two ends spiral around $\alpha$; Figure~\ref{Fig1} shows the situation in the surface $S$.

\begin{figure}[htbp]
\SetLabels
(.205*.68) $\beta$ \\
(.168*-.05) $\alpha$ \\
(.575*.68) $\alpha^m\beta$ \\
(.93*.68) $\alpha^{\infty}\beta$ \\
(.57*.85) $m=5$ \\
(.95*.85) $m=\infty$ \\
\endSetLabels
\centerline{\AffixLabels{\includegraphics[width=12cm, height=6cm]{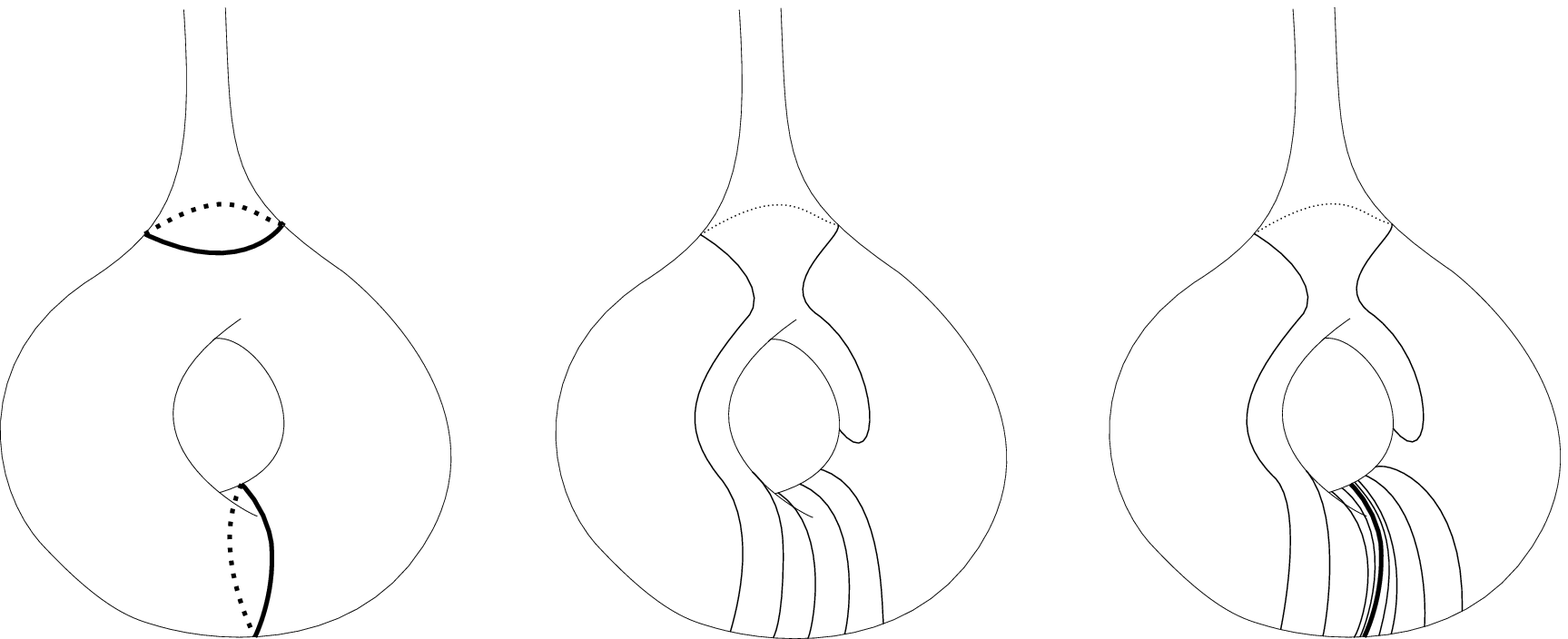}}}
\caption{}
\vskip 5pt
\label{Fig1}
\end{figure}

More precisely, let $D\subset T^1S$ be a small surface transverse to the geodesic foliation $\mathcal{F}$ that intersects the closed leaf $\alpha^*$ in one point $x^\infty_\infty$, as shown on Figure~\ref{Fig2}. The infinite leaf ${\alpha^\infty\beta}^*$ intersects $D$ in two sequences of points $x^\infty_1$, $x^\infty_2$, \dots{} and $y^\infty_1$, $y^\infty_2$, \dots{} in such a way that $x^\infty_1$, $x^\infty_2$, \dots{} converges in this order to one end of $\alpha^\infty\beta$, and $y^\infty_1$, $y^\infty_2$, \dots{} converges in this order to the other end. Moreover, the two sequences $x^\infty_1$, $x^\infty_2$, \dots{} and $y^\infty_1$, $y^\infty_2$, \dots{} both converge to the point $x^\infty_\infty$.

Likewise, the closed leaf ${\alpha^m\beta}^*$ intersects $D$ in points $x^m_1$, $x^m_2$, \dots{}, $x^m_{k_m}$, $y^m_1$, $y^m_2$, \dots{}, $y^m_{l_m}$ (see Figure 2), in such a way that, as $m$ tends to $\infty$, each $x^m_k$ converges to $x^\infty_k$, and each $y^m_l$ converges to $y^\infty_l$. Besides, the total number  $k_m+l_m$ of points is of the order of $m$, and more precisely, the difference $m-(k_m+l_m)$ is equal to a constant $c_D$ for $m$ large enough. As a result, if $\varphi$ is a continuous function defined on $D$, 
\begin{align*}
\lim_{m\to\infty}\frac{1}{m}\alpha^m\beta(\varphi)-\alpha(\varphi)
&=\lim_{m\to\infty}\frac{1}{m}\Biggl ( \sum_{i=1}^{k_m}\varphi(x^m_i) + \sum_{j=1}^{l_m}\varphi(y^m_j)\\
&\qquad\qquad\qquad -(k_m+l_m+c_D) \varphi(x^\infty_\infty) \Biggr) \\
&=0.
\end{align*}

\begin{figure}[htbp]
\SetLabels
(.6*.87) $D$ \\
(.18*.63) $\alpha$ \\
(.92*1) $\alpha^m\beta$ \\
(.30*.9) $m=4$ \\
(.52*.71) $x_{\infty}^{\infty}$ \\
(.53*.73) $x_{2}^{4}$ \\
(0.51*.78) $x_{1}^{4}$ \\
(.53*.67) $y_{2}^{4}$ \\
(.49*.61) $y_{1}^{4}$ \\
(.30*.33) $m=\infty$ \\
(.25*.18) $\alpha$ \\
(.6*.39) $D$ \\
(.53*.21) $x_{\infty}^{\infty}$ \\
(0.52*.32) $x_{1}^{\infty}$ \\
(0.54*.26) $x_{2}^{\infty}$ \\
(0.61*.24) $x_{3}^{\infty}$ \\
(0.5*.13) $y_{1}^{\infty}$ \\
(0.59*.21) $y_{3}^{\infty}$ \\
(0.56*.18) $y_{2}^{\infty}$ \\
\endSetLabels
\centerline{\AffixLabels{\includegraphics[width=10cm, height=10cm]{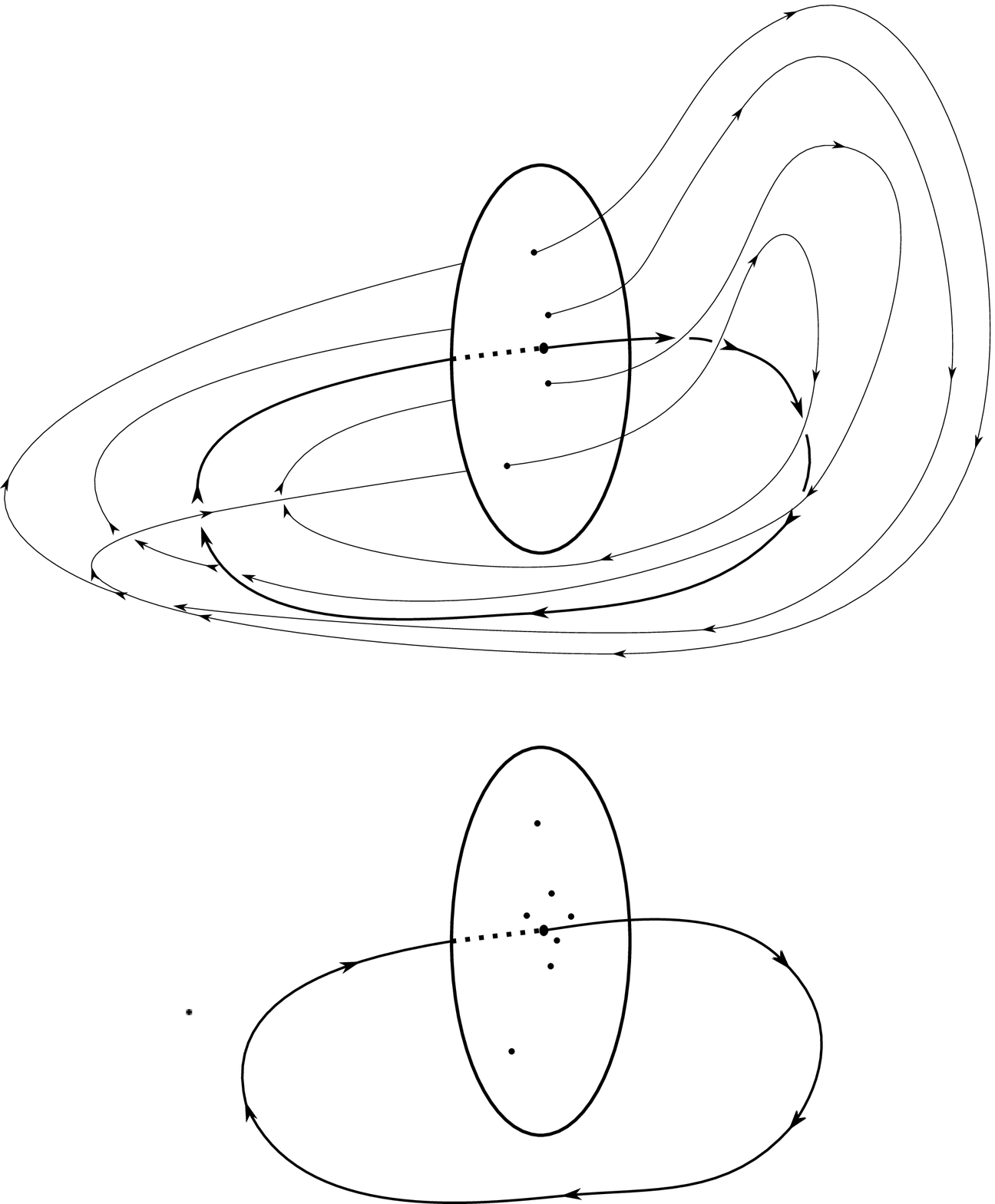}}}
\caption{}
\vskip 5pt
\label{Fig2}
\end{figure}

If, in addition, we assume $\varphi$ to be H\"older continuous, we can refine the above estimate. Because of the compacity of $S$, classical hyperbolic estimates guarantee that there exists some bound $M>0$ such that, for every $m\geq M$, both sequences $x^m_{i}$ and $y^m_i$ converge uniformely to $x^\infty_\infty$ as $i$ tends to $\infty$, and this convergence is of exponential order. Therefore
\begin{align*}
\lim_{m\to+\infty}\left ( \alpha^m\beta-m\alpha \right )(\varphi)&=\lim_{m\to+\infty} \sum_{i=1}^{k_m}\bigl ( \varphi(x^{m}_i)-\varphi(x^\infty_\infty) \bigr ) \\
&\qquad \qquad \qquad+\sum_{j=1}^{l_m}\bigl ( \varphi(y^{m}_j)-\varphi(x^\infty_\infty) \bigr )\\
&\qquad \qquad \qquad \qquad \qquad-c_D\varphi(x_\infty^\infty)\\
&=\sum_{i=1}^{\infty}\bigl ( \varphi(x^{\infty}_i)-\varphi(x^\infty_\infty) \bigr ) \\
& \qquad \qquad \qquad +\sum_{j=1}^{\infty}\bigl ( \varphi(y^{\infty}_j)-\varphi(x^\infty_\infty) \bigr )\\
& \qquad \qquad \qquad \qquad \qquad-c_D\varphi(x_\infty^\infty)
\end{align*}
exists and is finite.

In other words, the above calculation shows that the limit
$$
\lim_{m\to\infty}  \alpha^m\beta-m\alpha = \dot\alpha
$$
exists in the space of H\"older geodesic currents $\CH(S)$. The limit H\"older geodesic current $\dot\alpha$ is supported in the union of the closed leaf  $\alpha^*$, and of the infinite leaf ${\alpha^\infty\beta}^*$ whose two ends spiral around $\alpha^*$. 

Thus, by linearity and continuity of the length functions $\ell_i^\rho$, 
$$
\lim_{m \to \infty}\ell_{i}^{\rho}(\alpha^m\beta)-m\ell_{i}^{\rho}(\alpha)
$$
exists and is equal to $\ell_{i}^{\rho}(\dot\alpha)$. Taking the exponential on both sides, we conclude that  
$$
\frac{\lambda^\rho_i(\alpha^m \beta)}{\lambda^\rho_i(\alpha)^m}
$$ 
converges to $\mathrm e^{\ell_i^\rho(\dot\alpha)}$, which proves the required result.
\end{proof}

We conclude this section with one last observation. 
$$
\dot\alpha=\lim_{m\to\infty} \alpha^m\beta-m\alpha=\lim_{m\to\infty} \frac{\frac 1m \alpha^m\beta-\alpha}{\frac 1m}
$$
is very reminiscent of the expression of a derivative. In fact, there exists a $1$--parameter family of H\"older geodesic currents $\alpha_t \in \CH(S)$, $t \in \left [0, \varepsilon \right ]$, such that $\alpha_{0}=\alpha$, $\alpha_{\frac{1}{m}}=\frac{1}{m}\alpha^m \beta$, and $\dot{ \alpha}=\frac{d}{dt^+}{\alpha_{t}}_{|t=0}$. As a result, the distribution $\dot\alpha$ should be understood as a tangent vector at the point $\alpha$, and the above estimate comes as a consequence of the first order approximation
$$
\ell_i^{\rho}(\alpha_t)\approx\ell_i^{\rho}(\alpha_0)+t\frac{d}{d t^+}{\ell_i^{\rho}(\alpha_{t})}_{|t=0}
$$
where $\frac{d}{d t^+}{\ell_i^{\rho}(\alpha_{t})}_{|t=0}=\ell_i^{\rho}(\dot\alpha)$ by the previous facts. This should be compared with the differentiability property of the length functions $\ell^\rho_i$ of Corollary~\ref{Differentiability}.

\begin{rem}
In a preprint \cite{PoS} following the publication of an earlier version of the present paper, M. Pollicott and R. Sharp proposed an alternative proof of the estimate of Theorem~\ref{pro:application} that is based on symbolic dynamics. Their method enables to improve the result. They also prove asymptotic growth estimates for the length functions $\ell^\rho_i$. In particular, their approach makes crucially use of the $1$--forms $\omega_i$  introduced in \S\ref{sect:Forms}.
\end{rem}

\end{document}